\documentclass{amsart}
\usepackage{tikz-cd}
\usepackage{latexsym,amssymb}
\usepackage[english]{babel}
\usepackage{bm}
\usepackage{amsfonts, amsthm}
\usepackage{amsmath}
\usepackage{mathrsfs}
\usepackage{multirow}
\usepackage{bm}
\usepackage{pgf,tikz}
\usepackage{hyperref}

\newtheorem{prop}{Proposition}[section]
\newtheorem{lem}[prop]{Lemma}
\newtheorem{cor}[prop]{Corollary}
\newtheorem{thm}[prop]{Theorem}
\newtheorem{conj}[prop]{Conjecture}
\theoremstyle{definition}
\newtheorem{rem}[prop]{Remark}
\newtheorem{defi}[prop]{Definition}
\newtheorem{ex}[prop]{Example}

\def\Z{\mathbb{Z}}

\numberwithin{equation}{section}

\def\S{\mathbb{S}}
\def\Z{\mathbb{Z}}
\def\N{\mathbb{N}}

\def\G{\Gamma}

\def\H{\mathrm{H}}
\def\N{\mathrm{N}}

\def\S{\mathcal{S}}
\def\E{\mathcal{E}}
\def\D{\mathcal{D}}
\def\B{\mathcal{B}}
\def\F{\mathcal{F}}

\begin{document}

\title{Non-zero sum Heffter arrays and their applications}

\author[S. Costa]{Simone Costa}
\address{DICATAM - Sez. Matematica, Universit\`a degli Studi di Brescia, Via
Branze 43, I-25123 Brescia, Italy}
\email{simone.costa@unibs.it}

\author[S. Della Fiore]{Stefano Della Fiore}
\address{DII, Universit\`a degli Studi di Brescia, Via
Branze 38, I-25123 Brescia, Italy}
\email{s.dellafiore001@unibs.it}

\author[A. Pasotti]{Anita Pasotti}
\address{DICATAM - Sez. Matematica, Universit\`a degli Studi di Brescia, Via
Branze 43, I-25123 Brescia, Italy}
\email{anita.pasotti@unibs.it}

\begin{abstract}
In this paper we introduce a new class of partially filled arrays that, as Heffter arrays, are related to
difference families, graph decompositions and biembeddings.
A \emph{non-zero sum Heffter array} $\N\H(m,n; h,k)$ is an $m \times n$ p. f. array with entries in $\Z_{2nk+1}$ such that:
 each row contains $h$ filled cells and each column contains $k$ filled cells;
 for every $x\in \Z_{2nk+1}\setminus\{0\}$, either $x$ or $-x$ appears in the array;
 the sum of the elements in every row and column is different from $0$ (in $\Z_{2nk+1}$).
Here first we explain the connections with relative difference families and with path decompositions of the complete multipartite graph.
Then we present a complete solution for the existence problem and
a constructive complete solution for the square case and for the rectangular case with no empty cells when
the additional, very restrictive, property
of ``globally simple'' is required. Finally, we show how these arrays can be used to construct biembeddings of complete graphs.
\end{abstract}

\keywords{Heffter array, orthogonal cyclic path decomposition,  complete multipartite graph, embedding}
\subjclass[2010]{05B20; 05C10; 05C38}

\maketitle

\section{Introduction}
An $m\times n$ partially filled (p.f., for short) array on a set $\Omega$ is an $m \times n$ matrix whose  elements belong to $\Omega$
and where some cells can be empty. In 2015 Archdeacon \cite {A} introduced a class of p.f.  arrays which has been extensively studied:
the \emph{Heffter arrays}.
\begin{defi}\label{def:H}
A \emph{Heffter array} $\H(m,n; h,k)$ is an $m \times n$ p. f. array with entries in $\Z_{2nk+1}$ such that:
\begin{itemize}
\item[(\rm{a})] each row contains $h$ filled cells and each column contains $k$ filled cells;
\item[(\rm{b})] for every $x\in \Z_{2nk+1}\setminus\{0\}$, either $x$ or $-x$ appears in the array;
\item[(\rm{c})] the elements in every row and column sum to $0$ (in $\Z_{2nk+1}$).
\end{itemize}
\end{defi}
Trivial necessary conditions for the existence of a $\H(m,n; h,k)$ are $mh=nk$, $3\leq h\leq n$ and $3\leq k
\leq m$. If $m=n$ then $h=k$ and a square $\H(n,n;k,k)$ will be simply denoted by $\H(n;k)$.
The following result has been proved in \cite{ADDY, CDDY, DW}.
\begin{thm}\label{thm:Heffter}
 A $\H(n;k)$ exists for every $n\geq k\geq 3$.
 \end{thm}

In \cite{CPEJC} this concept has been generalized as follows.

\begin{defi}\label{def:lambdaRelative}
Let $v=\frac{2nk}{\lambda}+t$ be a positive integer,
where $t$ divides $\frac{2nk}{\lambda}$,  and
let $J$ be the subgroup of $\Z_{v}$ of order $t$.
 A  $\lambda$-\emph{fold Heffter array $A$ over $\Z_{v}$ relative to $J$}, denoted by $^\lambda\H_t(m,n; h,k)$, is an $m\times n$ p.f.  array
 with elements in $\Z_{v}$ such that:
\begin{itemize}
\item[($\rm{a_1})$] each row contains $h$ filled cells and each column contains $k$ filled cells;
\item[($\rm{b_1})$] the multiset $\{\pm x \mid x \in A\}$ contains  each element of $\Z_v\setminus J$ exactly $\lambda$ times;
\item[($\rm{c_1})$] the elements in every row and column sum to $0$ (in $\Z_v$).
\end{itemize}
\end{defi}
It is easy to see that if $\lambda=t=1$ we find again the arrays of Definition \ref{def:H}.

Classical Heffter arrays and the above generalization have been introduced also in view of their vast variety of applications and
connections with other much studied problems and concepts. In particular there are some recent papers in which Heffter arrays
are investigated to obtain new face $2$-colorable embeddings (briefly biembeddings) see \cite{A, CDY, CMPPHeffter,  CPPBiembeddings, CPEJC, DM}.
On the other hand several authors focused their attention on the existence problem, see \cite{ABD, ADDY, BCDY, CDDY, RelH,DW,MP,MP2,MP3}.

Here we introduce a new class of p.f. arrays, which is related to the one of Heffter arrays.

\begin{defi}\label{def:NZS}
Let $v=\frac{2nk}{\lambda}+t$ be a positive integer,
where $t$ divides $\frac{2nk}{\lambda}$,  and
let $J$ be the subgroup of $\Z_{v}$ of order $t$.
 A  \emph{non-zero sum} $\lambda$-\emph{fold Heffter array $A$ over $\Z_{v}$ relative to $J$}, denoted by $^\lambda \N\H_t(m,n; h,k)$, is an $m\times n$ p.f.  array
 with elements in $\Z_{v}$ such that:
\begin{itemize}
\item[($\rm{a_2})$] each row contains $h$ filled cells and each column contains $k$ filled cells;
\item[($\rm{b_2})$] the multiset $\{\pm x \mid x \in A\}$ contains each element of $\Z_v\setminus J$ exactly $\lambda$ times;
\item[($\rm{c_2})$] the sum of the elements in every row and column is different from $0$ (in $\Z_v$).
\end{itemize}
\end{defi}
Here we focus our attention on the case $\lambda=1$, namely when the entries of the array are pairwise distinct.
If the array is square then $h=k\geq 1$, while if $m\neq n$ then at least one of $h$ and $k$ has to be greater than $1$.
As done for Heffter arrays, a square non-zero sum Heffter array will be simply denoted by $\N\H(n;k)$.

This definition is motivated by Alspach's partial sums conjecture that deals with sets whose sum is different from $0$. In particular the truth of this conjecture would show that any $\N\H_t(m,n; h,k)$ provides two orthogonal path decompositions of $K_{\frac{2nk+t}{t}\times t}$.

The relations between non-zero sum Heffter arrays, Alspach's conjecture and orthogonal path decompositions will be explained in detail in Section \ref{sec:DF}. Then we will analyze the existence problem for $\N\H_t(m,n; h,k)$. In particular, in Section \ref{sec:square}, we will show a construction that completely solves the square case when $t=1$ (but, assuming the existence of a $\H_t(n;k)$, it works also for $t>1$) and, in Section \ref{sec:rectangular}, using a probabilistic approach, we will provide a complete solution to the general case. 
Moreover, in Section \ref{Section:Diagonal}, we will construct $k$-diagonal square non-zero sum Heffter arrays and rectangular ones with no empty cells that satisfy the very restrictive property of being ``globally simple''.
Finally, in the last section, we will show how these arrays can be used to construct $2$-colorable embeddings of complete graphs and we will prove some existence theorems about biembeddings.

\section{Relations with orthogonal path decompositions}\label{sec:DF}
In this section we explain the relation between non-zero sum Heffter arrays, difference families and orthogonal decompositions.
To do this we have to introduce some concepts and notation.

Given  a finite subset $T$ of an abelian group $G$ and an ordering $\omega=(t_1,t_2,\ldots,t_k)$ of the elements
in $T$,  let $s_i=\sum_{j=1}^i t_j$, for any $i=1,\ldots,k$,
be the $i$-th partial sum of $T$.
The ordering $\omega$ is said to be \emph{simple} if $s_b\neq s_c$ for all $1\leq b <  c\leq k$ or,
equivalently,
if there is no proper subsequence of $\omega$ that sums to $0$.
If $\omega$ is a simple ordering then also $\omega^{-1}=(t_k,t_{k-1},\ldots,t_1)$ is simple.
We point out that there are several problems and conjectures about distinct partial sums.
For instance, several years ago Alspach made the following conjecture, whose validity would shorten some cases of known proofs about the existence of cycle decompositions.

\begin{conj}\label{Conj:als}
Let $A\subseteq \mathbb{Z}_v\setminus\{0\}$ such that
$\sum_{a\in A}a\neq0$. Then there exists an ordering of the elements of $A$ such that the partial sums are all
distinct and non-zero.
\end{conj}
Conjecture \ref{Conj:als} has been proved for subsets $A$ of size $k$ of $\Z_v$ whenever
\begin{itemize}
\item[(1)] $k\leq 9$ or $k=10$ and $v$ is a prime (see \cite{A, HOS});
\item[(2)] $k\in \{v-1,v-2\}$ or $k=v-3$ and $v$ is a prime (see \cite{A, BH, HOS});
\item[(3)] $v\leq 25$ (see \cite{ADMS});
\item[(4)] the prime power factors of $v$ are greater than a suitable constant $N$ and $k\leq 11$ (see \cite{CP}).
\end{itemize}

Other problems related to the previous one have been proposed in
 \cite{AL, ADMS, CMPPSums, O}.

Given an $m\times n$ p.f.  array $A$,
the rows and the columns of $A$ are denoted by ${R}_1,\ldots,{R}_m$ and by ${C}_1,\ldots,{C}_n$, respectively.
We denote by $\E(A)$ the list of the elements of the filled cells of $A$
and by $\E({R}_i)$, $\E({C}_j)$, the elements  of the $i$-th row, of the $j$-th column, respectively, of $A$.
Also, by $\omega_{{R}_i}$ and $\omega_{{C}_j}$ we denote, respectively,  an ordering of $\E({R}_i)$ and of $\E({C}_j)$.  If for any $i=1,\ldots, m$
and for any $j=1,\ldots,n$, the orderings $\omega_{{R}_i}$ and $\omega_{{C}_j}$ are simple, we define
by $\omega_r=\omega_{{R}_1}\circ \ldots \circ\omega_{{R}_m}$ the simple ordering for the rows and
  by $\omega_c=\omega_{{C}_1}\circ \ldots \circ\omega_{{C}_n}$ the simple ordering for the columns.
Also, by \emph{natural ordering} of a row (column) of $A$ one means the ordering from left to right (from top to bottom).
  A p.f. array $A$ on an abelian  group $G$ is said to be
  \begin{itemize}
  \item \emph{simple} if there exists a simple ordering for each row and each column of $A$;
  \item \emph{globally simple} if the natural ordering of each row and each column of $A$ is simple.
  \end{itemize}
  \begin{rem}\label{Simple}
We note that, if Conjecture \ref{Conj:als} would be true, then any $\N\H_t(m,n;h,k)$ would be simple.
  Moreover, if $h,k\leq 3$, then every non-zero sum Heffter array is globally simple.
\end{rem}

 Given a graph $\G$, we denote by $V(\G)$ and $E(\G)$ the vertex-set and the edge-set of $\G$, respectively,
and by $^\lambda \G$ the multigraph obtained from $\G$ by repeating each edge $\lambda$ times.
Also by $K_v$, $K_{q \times r}$, $C'_k$, $P_k$ we represent the complete graph on $v$ vertices,
the complete multipartite graph with $q$ parts each of size $r$, the cycle of length $k$ and the path of length $k$
on $k+1$ distinct vertices, respectively.
Given a subgraph $\Gamma$ of a graph $K$, a $\G$-decomposition of $K$ is a set of graphs (said blocks), all isomorphic to $\G$,
whose edge-sets partition the edge-set of $K$, see for instance \cite{BEZ}.
Given an additive group $G$, a $\G$-decomposition $\D$ of a graph $K$ is $G$-regular if, up to isomorphisms,
$V(K)=G$ and for any $B\in {\D}$ also the graph $B+g$ is a block of ${\D}$ for any $g\in G$.
Here we will work with cyclic $P_k$-decompositions, namely decompositions into paths regular under the cyclic group.
It is well known that difference families are a very useful tool for constructing regular decompositions.
We  recall the definition,  see \cite{AB,B98}.
\begin{defi}
Let $\G$ be a graph with vertices in an additive group $G$. The multiset
$\Delta \G=\{\pm(x-y) \mid \{x,y\}\in E(\G)\}$
is called the \emph{list of differences} from $\G$.
\end{defi}
More generally, given a set $\mathcal{W}$ of graphs with vertices in $G$, by $\Delta\mathcal{W}$ one
means the union (counting multiplicities) of all multisets $\Delta\G$, where $\G\in \mathcal{W}$.

\begin{defi}\label{def:DF}
  Let $J$ be a subgroup of an additive group $G$ and let $\G$ be a graph.
  A collection $\F$ of graphs isomorphic to $\G$ and with vertices in $G$
  is said to be a $(G,J,\G,\lambda)$-\emph{difference family} (briefly, DF) \emph{over $G$ relative to} $J$
if each element of $G\setminus J$ appears exactly $\lambda$ times in the list of differences of $\F$
while no element of $J$ appears there.
\end{defi}

If $\lambda=t=1$ one speaks of a $(G,\G,1)$-DF.
Moreover, if $t$ is a divisor of $v$, a $(\Z_v,\frac{v}{t}\Z_v,\G,\lambda)$-DF, where $\frac{v}{t}\Z_v$ denotes the subgroup of $\Z_v$
of order $t$, is simply denoted by $(v,t,\G,\lambda)$-DF.
The connection between relative difference families and decompositions of a complete multipartite multigraph
is given by the following result.

\begin{prop}\label{thm:basecycles}\cite[Proposition 2.6]{BP}
If $\F=\{B_1,\ldots,B_\ell\}$ is a $(G,J,\G,\lambda)$-DF, then $\B=\{B_i+g \mid i=1,\ldots,\ell; g\in G\}$
is a $G$-regular $\G$-decomposition of $^\lambda K_{q\times r}$, where $q=|G:J|$ and $r=|J|$.
Hence, if $\F$ is a $(G,\G,1)$-DF there exists a $G$-regular $\G$-decomposition of $K_{|G|}$.
\end{prop}

 It is well know, see for instance \cite{CMPPHeffter}, that if there
exists a simple $\H_t(m,n;h,k)$, then
there exist a $(2mh+t,t,C'_h,$ $1)$-DF
and  a $(2nk+t,t,C'_k,1)$-DF. In a similar way one can prove the following.

\begin{prop}\label{from Heffter to DF}
If there exists a simple $\N\H_t(m,n;h,k)$, then
there exist a $(2mh+t,t,P_h,1)$-DF
and  a $(2nk+t,t,P_k,1)$-DF.
\end{prop}

\begin{rem}\label{rem:ortho}
Let $\F_h$ and $\F_k$ be the relative DFs constructed in the previous proposition.
  Note that for any $P_h\in \F_h$ and any $P_k\in \F_k$, we have
$|\Delta P_h \cap \Delta P_k| \in \{0,2\}.$
\end{rem}

We recall the following definition, see for instance \cite{CY1}.
\begin{defi}
  Two graph decompositions $\D$ and $\D'$ of a simple graph $K$ are said \emph{orthogonal} if and only if for any $B$ of $\D$
  and any $B'$ of $\D'$, $B$ intersects $B'$ in at most one edge.
\end{defi}

Heffter arrays are also used for getting orthogonal cycle decompositions, see for instance \cite{BCP}. Analogously from non-zero sum Heffter arrays one can construct orthogonal path decompositions as shown by the following result.

\begin{prop}\label{HeffterToDecompositions}
  Let $A$ be a $\N\H_t(m,n;h,k)$ simple with respect to the orderings $\omega_r$ and $\omega_c$. Then:
  \begin{itemize}
    \item[(1)] there exists a cyclic $h$-path decomposition $\D_{\omega_r}$ of $K_{\frac{2mh+t}{t}\times t}$;
    \item[(2)] there exists a cyclic $k$-path decomposition $\D_{\omega_c}$ of $K_{\frac{2nk+t}{t}\times t}$;
    \item[(3)] the decompositions $\D_{\omega_r}$ and $\D_{\omega_c}$ are orthogonal.
  \end{itemize}
\end{prop}

\begin{proof}
(1) and (2) follow from Propositions \ref{thm:basecycles} and \ref{from Heffter to DF}.
Then (3) follows from Remark \ref{rem:ortho}.
\end{proof}

We point out that the notion of simple array can be extended also to $^\lambda \N\H_t(m,n;h,k)$ with $\lambda>1$, similar to what was done in \cite{CPEJC} for Heffter arrays.
Then reasoning as above one can get two path decompositions of the $\lambda$-fold multipartite complete graph, but in this case the decompositions are not orthogonal.

\begin{ex}
Consider the $\N\H(3,4;4,3)$, say $A$, below.
\begin{center}
\begin{footnotesize}
$\begin{array}{|r|r|r|r|}
\hline 1 & 10 & -11 & 2 \\
\hline 8 & 6 & -3 & 5   \\
\hline  -7  & 12 &  9 & 4  \\\hline
\end{array}$
\end{footnotesize}
\end{center}
Firstly we find simple orderings for the rows and the columns of $A$. Note that, for instance, the natural
orderings from left to right of the first row and of the third row are not simple.
While every ordering of the columns is simple since $k=3$.

Set, for instance:
\begin{align*}
\omega_1&=(1,2,10,-11), & \nu_1&=(1,8,-7), \\
\omega_2&=(8,6,-3,5), & \nu_2&=(10,6,12), \\
 \omega_3&=(4,-7,12,9), &  \nu_3&=(-11,-3,9), \\
& & \nu_4&=(2,5,4).
\end{align*}
Starting from these orderings we obtain the following paths:
\begin{align*}
P_{\omega_1}&=[0,1,3,13,2], & P_{\nu_1}&=[0,1,9,2],\\
 P_{\omega_2}&=[0,8,14,11,16], & P_{\nu_2}&=[0,10,16,3], \\
 P_{\omega_3}&=[0,4,22,9,18], & P_{\nu_3}&=[0,14,11,20], \\
 & & P_{\nu_4}&=[0,2,7,11].
\end{align*}
Set $\F_{\omega}=\{P_{\omega_i}\mid i=1,\ldots,3\}$ and $\F_{\nu}=\{P_{\nu_i}\mid i=1,\ldots,4\}$;
by the construction of the paths it immediately follows that
$\Delta \F_{\omega}= \Delta \F_{\nu}= \Z_{25}\setminus \{0\}$.
Hence $\F_{\omega}$ is a $(25,P_4,1)$-DF and $\F_{\nu}$ is a $(25,P_3,1)$-DF.
Set ${\D}_{\omega}=\{P_{\omega_i}+g\mid i=1,\ldots,3, g\in \Z_{25}\}$
and ${\D}_{\nu}=\{P_{\nu_i}+g\mid i=1,\ldots,4, g\in \Z_{25}\}$.
Then $\D_{\omega}$ is a cyclic $P_4$-decomposition of $K_{25}$,
$\D_{\nu}$ is a cyclic $P_3$-decomposition of $K_{25}$. Moreover, $\D_{\omega}$ and $\D_{\nu}$ are orthogonal.
\end{ex}

\section{A constructive complete solution for the square case}\label{sec:square}
In this section we present a complete solution for the existence of a $\N\H(n;k)$ for any $n\geq k\geq 1$.

\begin{rem}\label{rem:k1}
The existence of a globally simple $\N\H(n;1)$, say $A=(a_{i,j})$, for every $n\geq 1$ is trivial.
In fact it is sufficient to set $a_{i,i}=i$ for every $i=1,\ldots,n$, while all other cells are empty.
\end{rem}

\begin{prop}\label{prop:k2}
  For every $n\geq2$, there exists a globally simple $\N\H(n;2)$.
\end{prop}
\begin{proof}
Let $A=(a_{i,j})$ be an $n\times n$ matrix with $n\geq 2$.
Set $a_{i,i}=i$ and $a_{i,i+1}=n+i$ for every $i=1,\ldots,n$, where the subscript are considered modulo $n$.
It is immediate to see that $A$ is a $\N\H(n;2)$.
\end{proof}

We recall that given a square matrix $A$ of order $n$ a \emph{transversal} of $A$ is a set of $n$ filled cells such that
no two belong to the same row and no two belong to the same column.

\begin{lem}\label{lem:trans}
A square p.f. array with $k\geq 1$ filled cells in each row and in each column
admits a transversal.
\end{lem}
\begin{proof}
Let $A$ be an $n \times n$ array as in the statement.
Associate to $A$ a bipartite graph, say $\Gamma$, on $2n$ vertices,
in which the vertices of one part represent the $n$ rows
and the vertices of the other part represent the $n$ columns of $A$ and there is an edge connecting two vertices if and only if the cell of the corresponding
row and column is not an empty cell. Note that a transversal of $A$ is nothing but a perfect matching of $\Gamma$ which, by construction, is $k$-regular and
it is well-known that a regular and bipartite graph admits a perfect matching.
\end{proof}

\begin{thm}\label{thm:darelative}
If there exists a $\H_t(n;k)$ then there also exists a $\N\H_t(n;k)$.
\end{thm}
\begin{proof}
Let $A$ be a  $\H_t(n;k)$. 
By Lemma \ref{lem:trans}, there exists a transversal $T$ of $A$ containing $n$ filled cells.
Let $H$ be the matrix obtained from $A$ by changing the signs of the elements of $T$ and leaving all other cells unchanged.
Note that to prove that $H$ is a $\N\H_t(n;k)$ we have only to check condition ($\rm{c_2})$ of Definition \ref{def:NZS}.
Now we remark that if $t$ is odd, $\Z_{2nk+t}$ does not contain an involution, while
if $t$ is even the involution of $\Z_{2nk+t}$ belongs to the subgroup of order $t$ of $\Z_{2nk+t}$.
Hence, in any case, $A$ does not contain the involution.
Also, we recall that $A$ is a $\H_t(n;k)$ hence each row and each column of $A$ sums to zero. 
It follows that changing the sign of exactly one element in each row and in each column of $A$
all the sums are now different from zero. The result follows.
\end{proof}

\begin{cor}\label{thm:square}
For every $n\geq k\geq 1$ there exists a $\N\H(n;k)$.
\end{cor}
\begin{proof}
If $k=1$ or $k=2$ the result follows from Remark \ref{rem:k1} and Proposition \ref{prop:k2}, respectively.
For $k\geq 3$, the result follows from Theorems \ref{thm:Heffter} and \ref{thm:darelative}.
\end{proof}

In the next section we will present a generalization of the previous corollary, whose proof is not constructive.

\begin{ex}
Here we have the $\H(7;5)$ constructed in \cite{ADDY}. Note that the cells with bold elements
form a transversal of the array.
\begin{center}
\begin{footnotesize}
$\begin{array}{|r|r|r|r|r|r|r|}\hline
\bf{-10} &  & 16  & -1 & -2 & -3 &  \\
\hline   & \bf{-4} &  & -6 & -7 & -5 & 22   \\
\hline  -30 & 29 & \bf{-9} & -8 &   &  & 18  \\
\hline  -11 &   & -12 & \bf{28} & -31 & 26 &  \\
\hline   & -14 & -15 & -13 &  & \bf{17} & 25 \\
\hline  27 & -34 & 20 &   & \bf{19} &  & -32  \\
\hline  24 & 23 &  &  & 21 &  -35 & \bf{-33}  \\\hline
\end{array}$
\end{footnotesize}
\end{center}

Applying the  proof of Theorem \ref{thm:darelative} we get the $\N\H(7;5)$ below.
\begin{center}
\begin{footnotesize}
$\begin{array}{|r|r|r|r|r|r|r|}\hline
10 &  & 16  & -1 & -2 & -3 &  \\
\hline   & 4 &  & -6 & -7 & -5 & 22   \\
\hline  -30 & 29 & 9 & -8 &   &  & 18  \\
\hline  -11 &   & -12 & -28 & -31 & 26 &  \\
\hline   & -14 & -15 & -13 &  & -17 & 25 \\
\hline  27 & -34 & 20 &   & -19 &  & -32  \\
\hline  24 & 23 &  &  & 21 &  -35 & 33  \\\hline
\end{array}$
\end{footnotesize}
\end{center}
\end{ex}

\section{A complete solution for the rectangular case}\label{sec:rectangular}
In this section we prove the existence of a $\N\H_t(m,n;h,k)$ whenever the trivially necessary conditions
$m\geq k\geq 1$, $n\geq h\geq 1$, $nk=mh$ and $t | 2nk$ are satisfied.

\begin{lem}\label{ExistenceScheletro}
Let $m\geq k\geq 1$ and $n\geq h\geq 1$ be such that $nk=mh$. Then there exists an $m\times n$ p.f. array $A$ that has exactly $h$ filled cells in each row and exactly $k$ filled cells in each column.
\end{lem}
\begin{proof}
Let $m,n,h,k$ be as in the statement. Let $S$ be the subgroup of $\Z_m \times \Z_n$ generated by $(1,1)$.
It is not hard to see that the order of $S$ is equal to $\text{lcm}(m,n)$. Now consider the $m \times n$ empty array, say $A$,
whose rows represent the elements of $\Z_m$ while the columns represent the elements of $\Z_n$.
One can check that filling exactly the cells of $A$ correspoding to the elements of the subgroup $S$, we get a p.f. array with exactly
$\frac{\text{lcm}(m,n)}{m}$ filled cells in each row
and $\frac{\text{lcm}(m,n)}{n}$ filled cells in each column.
Now note that $\text{lcm}(m,n)$ divides $nk=mh$ and set $r=\frac{nk}{\text{lcm}(m,n)}$.
Consider $r$ distinct cosets of $S$ in $\Z_m \times \Z_n$. Filling exactly the cells of $A$ corresponding to the elements
of these cosets we obtain the required array.
 \end{proof}

\begin{thm}\label{Th:existence}
For every $m\geq k\geq 1$, $n\geq h\geq 1$ and $t$ such that $nk=mh$ and $t|2nk$ there exists a $\N\H_t(m,n;h,k)$.
\end{thm}
\begin{proof}
If $m=1$, which implies $k=1$ and $n=h$, take an arbitrary subset $S$ of $\Z_{2n+t}\setminus \frac{2n+t}{t}\Z_{2n+t}$ of size $n$
which does not contain pairs of the form $\{x,-x\}$ for some $x\in \Z_{2n+t}$.
Note that such a set $S$ exists since $\Z_{2n+t}\setminus \frac{2n+t}{t}\Z_{2n+t}$ does not contain the involution (if it exists).
Let $A$ be a $1\times n$ array filled with the elements of $S$. If the row of $A$ does not sum to zero, $A$
is a $\N\H_t(1,n;n,1)$, otherwise, it is sufficient to change the sign of an element in $A$ and the new array is a
 $\N\H_t(1,n;n,1)$. Clearly if $n=1$ one can reason in the same way.

 So in the following we can assume $m,n\geq 2$.
Let $S$ be a subset of $\Z_{2nk+t}\setminus \frac{2nk+t}{t}\Z_{2nk+t}$ of size $nk$ and such that, for every $x\in \Z_{2nk+t}\setminus \frac{2nk+t}{t}\Z_{2nk+t}$, exactly one of $x$ or $-x$ appears in $S$.
As before, such a set $S$ exists since $\Z_{2nk+t}\setminus \frac{2nk+t}{t}\Z_{2nk+t}$ does not contain the involution (if it exists).
We also consider an $m\times n$ p.f. array $A$ having exactly $h$ filled cells in each row and exactly $k$ filled cells in each column
whose existence follows by Lemma \ref{ExistenceScheletro}.
Now, we prove that we can write the elements of $S$ in the filled cells of $A$ 
 so that the sum over each row and column is non-zero.
For this purpose, we fill $A$, uniformly at random, with the elements of $S$.
Then we evaluate the expected value $\mathbb{E}(X)$ of the random variable $X$ given by the number of rows that sum to zero.

Due to the linearity of $\mathbb{E}(X)$ we have that:
$$\mathbb{E}(X)=\sum_{i=1}^m \mathbb{P}\left(\sum_{x\in {R}_i}x=0\right).$$
For symmetry,  $\mathbb{P}(\sum_{x\in {R}_i}x=0)=\mathbb{P}(\sum_{x\in {R}_1}x=0)$  for every $i=1,\ldots,m$, and hence
$$\mathbb{E}(X)=m \cdot \mathbb{P}\left(\sum_{x\in {R}_1}x=0\right).$$
Now we want to give an upper-bound for the term $\mathbb{P}(\sum_{x\in {R}_1}x=0)$.
We note that, if $h\in \{1,2\}$, this probability is zero since neither $0$ nor pairs of the form $\{x,-x\}$ are contained in $S$.
Therefore, we can suppose $h\geq 3$. Let us assume we have already chosen the first $h-1$ elements, $x_1,\dots,x_{h-1}$, of ${R}_1$.
Then there exists at most one $\bar{x}\in S\setminus\{x_1,\dots,x_{h-1}\}$ that makes the sum zero. This means that
$$\mathbb{P}\left(\sum_{x\in {R}_1}x=0\right)\leq \frac{1}{mh-(h-1)} $$
and hence
$$\mathbb{E}(X)\leq \frac{m}{mh-(h-1)}.$$

Note that, since $m\geq 2$ and $h\geq 3$,  we obtain:
\begin{eqnarray}\label{eq:E}
  2m+(h-1)\leq 2m+m(h-2) &=& mh.
\end{eqnarray}
We also emphasize that in (\ref{eq:E}) the equality holds only if $(m,h)=(2,3)$
which implies $(n,k)=(3,2)$ or $(n,k)=(6,1)$ since $mh=nk$ and $n\geq h=3$.
This implies that $\frac{m}{mh-(h-1)}\leq 1/2$, namely $\mathbb{E}(X)\leq 1/2$.
Similarly, also the expected value $\mathbb{E}(Y)$ of the random variable $Y$ given by the number of columns that sum to zero cannot exceed $1/2$.
By the previous remark on the equality in (\ref{eq:E}), it immediately follows that if one of $\mathbb{E}(X)$
or $\mathbb{E}(Y)$ is equal to $1/2$, then the other one is zero, which implies that $\mathbb{E}(X)+\mathbb{E}(Y)<1$.
Therefore we can fill $A$ with the elements of $S$ so that the sum over each row and column is non-zero.
\end{proof}

From Remark \ref{Simple}, Proposition \ref{HeffterToDecompositions} and Theorem \ref{Th:existence} we have the existence of a pair of orthogonal path decompositions whenever Conjecture \ref{Conj:als} is satisfied. In particular, from the results of \cite{A} and \cite{HOS}, it follows that:
\begin{prop}
Let $m\geq k\geq 1$, $n\geq h\geq 1$ and $t$ be such that $nk=mh$ and $t|2nk$. Then, assuming that either $k,h\leq 9$ or $k,h\leq 10$ and $2nk+t$ is a prime, there exist a cyclic $h$-path decomposition $\D_{\omega_r}$ of $K_{\frac{2mh+t}{t}\times t}$ and a cyclic $k$-path decomposition $\D_{\omega_c}$ of $K_{\frac{2nk+t}{t}\times t}$.
Also, $\D_{\omega_r}$ and $\D_{\omega_c}$ are a pair of orthogonal decompositions.
\end{prop}

\section{A constructive complete solution for globally simple $\N\H(n;k)$ and $\N\H(m,n;n,m)$}\label{Section:Diagonal}

In this section we construct a globally simple square non-zero sum Heffter array
$\N\H(n;k)$ for every $n\geq k\geq 1$ and a globally simple  non-zero sum Heffter array
$\N\H(m,n;n,m)$ for every $m,n\geq 1$.

Given an $n \times n$ p.f. array $A$, for any $1\leq i\leq n$, the $i$-th diagonal of $A$ is defined as follows:
$$D_i = \{(i,1),(i+1,2),\ldots, (i-1,n)\}.$$
All the arithmetic on the row and column indices is performed modulo $n$, where the set of reduced residues is $\{1,2,\ldots,n\}$.
The $k+1$ diagonals $D_i,D_{i+1}, \ldots, D_{i+k}$ are said to be \emph{consecutive diagonals}.

\begin{defi}
A square p.f. array $A$ of size $n\geq k \geq 1$ is said to be \emph{cyclically $k$-diagonal}
if the nonempty cells of $A$ are exactly those of $k$ consecutive diagonals.
\end{defi}

\begin{thm}\label{thm:gs}
For any $n\geq k\geq 1$ there exists a cyclically $k$-diagonal globally simple $\N\H(n;k)$.
\end{thm}
\begin{proof}
Let $n\geq k \geq 1$ and let $A=(a_{i,j})$ be the $n \times n$ array defined as follows.
If $j \geq i$, set $j=i+r$, hence $a_{i,j}=a_{i,i+r}$.
If $j<i$ instead of $a_{i,j}$ we can consider $a_{i,j+n}$ since, as usual, the subscripts are considered modulo $n$.
Then $j+n=i+r$ for some $r$ and hence, also in this case, $a_{i,j}$ can be written in the form $a_{i,i+r}$.
Now, for any $0\leq r\leq k-1$ we set $a_{i,i+r}=\varepsilon[r+1+(i-1)k]$ where $\varepsilon$ is equal to $1$ if $r$ is even,
while $\varepsilon$ is equal to  $-1$ for $r$ odd. All the other cells of $A$ are left empty.
One can check that the filled cells of $A$ are exactly those of the diagonals $D_1,D_2,\ldots,D_k$,
so $A$ is a cyclically $k$-diagonal array.
Now we want to prove that $A$ is a globally simple $\N\H(n;k)$.
It is easy to see that conditions ($\rm{a_2})$ and ($\rm{b_2})$ of Definition \ref{def:NZS} are satisfied.

We list the partial sums for each row starting from the element $a_{i,i}$, for $1\leq i \leq n$.
If $k$ is even we have
\begin{footnotesize}
\begin{eqnarray*}
\S(R_i) &=& \Big((i-1)k+1,-1,(i-1)k+2,-2,\ldots, (i-1)k+\frac{k}{2}, -\frac{k}{2}\Big),
\end{eqnarray*}
\end{footnotesize}
while for $k$ odd it results
\begin{footnotesize}
\begin{eqnarray*}
\S(R_i) &=& \Big((i-1)k+1,-1,(i-1)k+2,-2,\ldots, \frac{k-1}{2},(i-1)k+\frac{k+1}{2}\Big).
\end{eqnarray*}
\end{footnotesize}

So it is immediate that the sum of the elements of each row is different from zero.

Now we list the partial sums for each column. We start from the first $k-1$ columns, and we have four cases according to the parity of $i$ and $k$.
In each of these cases we list the sums starting from the element $a_{n-k+1+i,i}$.

Case 1. $k$ even, $i$ odd with $1\leq i \leq k-1$.
\begin{footnotesize}
\begin{eqnarray*}
\S(C_i) &=& \Big(-k(n-k+1+i), k-1,-(k-1)-k(n-k+1+i), 2(k-1),\\
& & -2(k-1)-k(n-k+1+i),3(k-1),\ldots,\\
& &  \frac{k-i-1}{2}(k-1), -\frac{k-i-1}{2}(k-1)-k(n-k+1+i),\\
& & -\frac{k-i-1}{2}(k-1)-k(n-k+1+i)+i, -\frac{k-i+1}{2}(k-1)-k(n-k+1+i),\\
& & -\frac{k-i+1}{2}(k-1)-k(n-k+1+i)+i, \ldots, -\frac{k-2}{2}(k-1)-k(n-k+1+i),\\
& & -\frac{k-2}{2}(k-1)-k(n-k+1+i)+i\Big).
\end{eqnarray*}
\end{footnotesize}

Case 2. $k$ even, $i$ even with $1\leq i \leq k-1$.
\begin{footnotesize}
\begin{eqnarray*}
\S(C_i) &=& \Big(-k(n-k+1+i), k-1,-(k-1)-k(n-k+1+i), 2(k-1),\\
& & -2(k-1)-k(n-k+1+i), 3(k-1), \ldots,-\frac{k-i-2}{2}(k-1)-k(n-k+1+i),\\
& &   \frac{k-i}{2}(k-1),\frac{k-i}{2}(k-1)-i, \frac{k-i+2}{2}(k-1),  \frac{k-i-2}{2}(k-1)-i,\\
& & \frac{k-i+4}{2}(k-1),\ldots,  \frac{k-2i+2}{2}(k-1)-i,\frac{k}{2}(k-1)\Big).
\end{eqnarray*}
\end{footnotesize}

Case 3. $k$ odd, $i$ odd with $1\leq i \leq k-1$.
\begin{footnotesize}
\begin{eqnarray*}
\S(C_i) &=& \Big(k(n-k+1+i), -(k-1),(k-1)+k(n-k+1+i), -2(k-1),\\
& & 2(k-1)+k(n-k+1+i), -3(k-1), \ldots,-\frac{k-i}{2}(k-1), -\frac{k-i}{2}(k-1)+i,\\
& &  -\frac{k-i+2}{2}(k-1), -\frac{k-i-2}{2}(k-1)+i,  -\frac{k-i+4}{2}(k-1), -\frac{k-i-4}{2}(k-1)+i,\\
& &  -\frac{k-i+6}{2}(k-1),\ldots,-\frac{k-2i+1}{2}(k-1)+i\Big).
\end{eqnarray*}
\end{footnotesize}

Case 4. $k$ odd, $i$ even with $1\leq i \leq k-1$.
\begin{footnotesize}
\begin{eqnarray*}
\S(C_i) &=& \Big(k(n-k+1+i), -(k-1),(k-1)+k(n-k+1+i), -2(k-1),\\
& & 2(k-1)+k(n-k+1+i), -3(k-1), \ldots,-\frac{k-i-1}{2}(k-1),\\
& & \frac{k-i-1}{2}(k-1)+k(n-k+1+i),  \frac{k-i-1}{2}(k-1)+k(n-k+1+i)-i,\\
& & \frac{k-i+1}{2}(k-1)+k(n-k+1+i),  \frac{k-i-3}{2}(k-1)+k(n-k+1+i)-i, \\
& &\frac{k-i+3}{2}(k-1)+k(n-k+1+i), \ldots, \frac{k-2i+1}{2}(k-1)+k(n-k+1+i)-i,\\
& &\frac{k-1}{2}(k-1)+k(n-k+1+i)\Big).
\end{eqnarray*}
\end{footnotesize}

Now we list the sums of the columns $C_j$ for $k\leq j\leq n$. Write $j$ as $k-1+i$ for $1\leq i \leq n-k+1$.
We have two cases according to the parity of $k$.

If $k$ is even we have
\begin{footnotesize}
\begin{eqnarray*}
\S(C_{k-1+i}) &=& \Big(-ik, k-1,-(k-1)-ik,2(k-1), -2(k-1)-ik,3(k-1),\ldots,\\
& & -\frac{k-2}{2}(k-1)-ik,\frac{k}{2}(k-1)\Big),
\end{eqnarray*}
\end{footnotesize}
while for $k$ odd it results
\begin{footnotesize}
\begin{eqnarray*}
\S(C_{k-1+i}) &=& \Big(ik, -(k-1),(k-1)+ik,-2(k-1), 2(k-1)+ik,-3(k-1),\ldots,\\
& & -\frac{k-1}{2}(k-1),\frac{k-1}{2}(k-1)+ik\Big).
\end{eqnarray*}
\end{footnotesize}

Note that also the sum of the elements of each column is different from zero.
Hence $A$ is a $\N\H(n;k)$. It can be easily verified that the partial sums for each row and each of the last $n-k+1$ columns are pairwise distinct modulo $2nk+1$. Instead proving that the partial sums for each of the first $k-1$ columns are pairwise distinct is slightly tedious. For this reason, the partial sums distinctness was verified with Mathematica. So, we defer the reader to the following link for the source code.

\begin{center}
\url{http://stefano-dellafiore.me/additional_material_NH.zip}.
\end{center}

\end{proof}

To help the reader to follow the proof of the previous theorem we present two examples
with different parity of $k$.

\begin{ex}
 Applying the  proof of Theorem \ref{thm:gs} we get the $\N\H(11;8)$ and the $\N\H(11;9)$ below.

\begin{center}
\begin{footnotesize}
$\begin{array}{|r|r|r|r|r|r|r|r|r|r|r|}
\hline 1 & -2 & 3  & -4 & 5 & -6 & 7 & -8 & & & \\
\hline  & 9 & -10  & 11 & -12 & 13 & -14 & 15 & -16 & & \\
\hline  &  & 17  & -18 & 19 & -20 & 21 & -22 & 23 & -24 & \\
\hline  &  &   & 25 & -26 & 27 & -28 & 29 & -30 & 31 & -32 \\
\hline -40 &  &   &  & 33 & -34 & 35 & -36 & 37 & -38 & 39 \\
\hline 47 & -48 &   &  &  & 41 & -42 & 43 & -44 & 45 & -46 \\
\hline -54 & 55 & -56  &  &  &  & 49 & -50 & 51 & -52 & 53 \\
\hline 61 & -62 & 63  & -64 & &  &  & 57 & -58 & 59 & -60 \\
\hline -68 & 69 & -70  & 71 & -72 &  &  &  & 65 & -66 & 67 \\
\hline 75 & -76 & 77  & -78 & 79 & -80 &  &  & & 73 & -74 \\
\hline -82 & 83 & -84  & 85 & -86 & 87 & -88 &  & & & 81 \\
\hline
\end{array}$
\end{footnotesize}
\end{center}

\begin{center}
\begin{footnotesize}
$\begin{array}{|r|r|r|r|r|r|r|r|r|r|r|}
\hline 1 & -2 & 3  & -4 & 5 & -6 & 7 & -8 & 9 & & \\
\hline   & 10  & -11 & 12 & -13 & 14 & -15 & 16 & -17 & 18 & \\
\hline  &  &  19 & -20 & 21 & -22 & 23 & -24 & 25 & -26 & 27 \\
\hline  36 &  &   & 28 & -29 & 30 & -31 & 32 & -33 & 34 & -35 \\
\hline -44 & 45 &   &  & 37 & -38 & 39 & -40 & 41 & -42 & 43 \\
\hline 52 & -53 & 54  &  &  & 46 & -47 & 48 & -49 & 50 & -51 \\
\hline -60 & 61 & -62  & 63 &  &  & 55 & -56 & 57 & -58 & 59 \\
\hline 68 & -69 & 70  & -71 & 72 &  &  & 64 & -65 & 66 & -67 \\
\hline  -76 & 77  & -78 & 79 & -80 & 81 &  & & 73 & -74 & 75 \\
\hline  84 & -85  & 86 & -87 & 88 & -89 & 90 & & & 82 & -83 \\
\hline  -92 & 93  & -94 & 95 & -96 & 97 & -98 & 99 & &  & 91 \\
\hline
\end{array}$
\end{footnotesize}
\end{center}

\end{ex}

\begin{prop}
For any $n\geq k\geq 1$, there exist a pair of orthogonal cyclic $P_k$-decompositions of $K_{2nk+1}$.
\end{prop}
\begin{proof}
The result follows by Proposition \ref{HeffterToDecompositions} and Theorem \ref{thm:gs}.
\end{proof}

\begin{thm}\label{thm:gsrectangular}
For any $m,n\geq 1$ there exists a globally simple $\N\H(m,n;n,m)$.
\end{thm}
\begin{proof}
Let $m,n \geq 1$ and let $A=(a_{i,j})$ be the $m \times n$ array defined as follows.
For any $1\leq i\leq m$ and $1 \leq j \leq n$ set $a_{i,j}=\varepsilon [j+(i-1)n]$ where $\varepsilon$ is equal to $1$ if $i+j$ is even,
while it holds $-1$ for $i+j$ odd.
We want to prove that $A$ is a globally simple $\N\H(m,n;n,m)$.
It is easy to see that conditions ($\rm{a_2})$ and ($\rm{b_2})$ of Definition \ref{def:NZS} are satisfied.
Also, one can check that the partial sums of the rows, written starting from the element $a_{i,1}$, are exactly the same
as the array constructed in the proof of Theorem \ref{thm:gs}. So we already know, that the sum
of the elements of each row is not zero and that the partial sums
are pairwise distinct.

Now we list the partial sums of the columns starting from $a_{1,j}$, for $1 \leq j \leq n$.
If $m$ is even we have
\begin{footnotesize}
\begin{eqnarray*}
\S(C_j)&=& \pm\left(j, -n, n+j, -2n,2n+j,-3n,3n+j, \ldots, \frac{m-2}{2}n+j,-\frac{m}{2}n\right),
\end{eqnarray*}
\end{footnotesize}
while for $m$ odd it results
\begin{footnotesize}
\begin{eqnarray*}
\S(C_{j}) &=& \pm\left(j, -n, n+j, -2n,2n+j,-3n,3n+j, \ldots, -\frac{m-1}{2}n+j,\frac{m-1}{2}n+j\right),
\end{eqnarray*}
\end{footnotesize}
where the sign is $+$ if $j$ is odd, $-$ if $j$ is even.
Note that also the sum of the elements of each column is different from zero.
Hence $A$ is a $\N\H(m,n;n,m)$. Finally, by a direct check, it is not hard to see that the partial sums of each column are pairwise distinct in $\Z_{2mn+1}$,
hence $A$ is a globally simple $\N\H(m,n;n,m)$.
\end{proof}

\begin{ex}
 Applying the  proof of Theorem \ref{thm:gsrectangular} we get the $\N\H(8,5;5,8)$ and the $\N\H(9,4;4,9)$ below.
\begin{footnotesize}
$$\begin{array}{|r|r|r|r|r|}
\hline 1 & -2 & 3  & -4 & 5 \\
\hline  -6 & 7 & -8 & 9 & -10 \\
\hline  11 & -12 & 13 & -14 & 15 \\
\hline  -16 & 17 & -18 & 19 & -20\\
\hline 21 & -22 & 23 & -24 & 25\\
\hline -26 & 27 & -28 & 29 & -30\\
\hline 31 & -32 & 33 & -34 & 35 \\
\hline  -36 & 37 & -38 & 39 & -40\\
\hline
\end{array}
\hspace{2cm}
\begin{array}{|r|r|r|r|}
\hline 1 & -2 & 3  & -4 \\
\hline  -5 & 6 & -7 & 8  \\
\hline  9 & -10 & 11  & -12 \\
\hline  -13 & 14  & -15 & 16\\
\hline 17 &  -18 & 19 & -20\\
\hline  -21 & 22 & -23 & 24\\
\hline  25 & -26 & 27 & -28 \\
\hline   -29 & 30 & -31 & 32\\
\hline   33 & -34 & 35 & -36\\
\hline
\end{array}$$
\end{footnotesize}
\end{ex}

\begin{prop}
For any $m,n\geq 1$, there exist a cyclic $P_m$-decomposition $\D_{\omega_r}$ of $K_{2mn+1}$
and a cyclic $P_n$-decomposition $\D_{\omega_c}$ of $K_{2mn+1}$. Also, $\D_{\omega_r}$ and $\D_{\omega_c}$ are a pair of orthogonal decompositions.
\end{prop}
\begin{proof}
The result follows by Proposition \ref{HeffterToDecompositions} and Theorem \ref{thm:gsrectangular}.
\end{proof}

\section{Relation with embeddings}\label{sec:embedding}
As already remarked in the Introduction, Archdeacon \cite{A} introduced Heffter arrays also because of their applications and, in particular, since they are useful for finding biembeddings of cycle decompositions.
In this section, generalizing some of Archdeacon's results we show how starting from a non-zero sum Heffter array it is possible to
obtain suitable biembeddings.

We recall the following definition, see \cite{Moh}.

\begin{defi}
An \emph{embedding} of a graph $\G$ in a surface $\Sigma$ is a continuous injective
mapping $\psi: \G \to \Sigma$, where $\G$ is viewed with the usual topology as $1$-dimensional simplicial complex.
\end{defi}

The connected components of $\Sigma \setminus \psi(\G)$ are called $\psi$-\emph{faces}.
If each $\psi$-face is homeomorphic to an open disc, then the embedding $\psi$ is said to be \emph{cellular}.

\begin{defi}
A \emph{biembedding} of two circuit  decompositions $\D$ and $\D'$ of a simple graph $\G$ is a face $2$-colorable embedding
of $\G$ in which one color class is comprised of the circuits in $\D$ and
the other class contains the circuits in $\D'$.
\end{defi}

Following the notation given in \cite{A}, for every edge $e$ of a graph $\G$, let $e^+$ and $e^-$ denote
its two possible directions
and let $\tau$ be the involution swapping $e^+$ and $e^-$ for every $e$.
Let $D(\G)$ be the set of all directed edges that correspond to edges of $\G$ and, for any $v\in V(\G)$, call $D_v$ the set of edges directed out
of $v$.
A \emph{local rotation} $\rho_v$ is a cyclic permutation of $D_v$. If we select a local rotation for each vertex of $\G$, then
all together
they form a rotation of $D(\G)$. We recall the following result, see \cite{A, GT, MT}.

\begin{thm}\label{thm:embedding}
A rotation $\rho$ on $\G$ is equivalent to a cellular embedding of $\G$ in an orientable surface.
The face boundaries of the embedding corresponding to $\rho$ are the orbits of $\rho \circ \tau$.
\end{thm}
Now our goal is to provide a construction of a biembedding starting from a non-zero sum Heffter array.
In order to characterize the faces of this embedding, we introduce some notations.

Given a cyclic path decomposition $\D$ of $K_{v}$ we want to define an associated circuit decomposition $\mathcal{C}(\D)$. As usual, we identify the vertex-set of $K_{v}$ with $\mathbb{Z}_{v}$.
Let us consider $P=(x_0,x_1,\dots,x_k)\in \D$, and let $\lambda_P$ be the minimum positive integer such that $\lambda_P(x_k-x_0)= 0$ in $\mathbb{Z}_{v}$. Then we set
$$C_P:=\bigcup_{i=0}^{\lambda_P-1} P+i(x_k-x_0).$$
Note that, because of the definition, $C_P$ is a circuit and we say that the path $P$ generates $C_P$.
Now we consider the set of all circuits of type $C_P$ without repeating them. More formally, we define
$$\mathcal{C}(\D):=\{C_P:\ P\in \D\}.$$
Since $\mathcal{C}(\D)$ is a set, chosen from a set $\mathcal{T}$ of generating paths for those circuits, we have that
$$\mathcal{C}(\D)=\{C_P:\ P\in \mathcal{T}\}.$$
Moreover, since $\bigcup_{P\in \D} P=\bigcup_{P\in \mathcal{T}} C_P$, we have that $\mathcal{C}(\D)$ is a decomposition of $K_{v}$ into circuits.

Then, adapting the proof of Theorem 3.4 of \cite{CPPBiembeddings}, we show that the Archdeacon embedding (see \cite{A}) is well-defined also in this case. For this purpose we recall the following definition. Given a non-zero sum Heffter array $A=\N\H(m,n;h,k)$, the orderings $\omega_r$ and $\omega_c$ are said to be \emph{compatible} if $\omega_c \circ \omega_r$ is a cycle of length $|\E(A)|$.

\begin{thm}\label{thm:biembedding}
Let $A$ be a non-zero sum Heffter array $\N\H(m,n;h,k)$ that is simple with respect to the compatible orderings $\omega_r$
and $\omega_c$.
Then there exists a cellular biembedding of the circuit decompositions $\mathcal{C}(\mathcal{D}_{\omega_r^{-1}})$ and
$\mathcal{C}(\mathcal{D}_{\omega_c})$ of $K_{2nk+1}$ into an orientable surface.
\end{thm}

\begin{proof}
Since the orderings $\omega_r$ and $\omega_c$ are compatible, we have that $\omega_c\circ\omega_r$ is a cycle of length
$|\E(A)|$. Let us consider the permutation $\bar{\rho}_0$ on $\pm \E(A)=\mathbb{Z}_{2nk+1}\setminus\{0\}$ defined by:
$$\bar{\rho}_0(a)=\begin{cases}
-\omega_r(a)\mbox{ if } a\in \E(A);\\
\omega_c(-a)\mbox{ if } a\in -\E(A).\\
\end{cases}$$
Note that, if $a\in \E(A)$, then $\bar{\rho}_0^2(a)=\omega_c\circ\omega_r(a)$ and hence $\bar{\rho}_0^2$ acts
cyclically on $\E(A)$. Also $\bar{\rho}_0$ exchanges $\E(A)$ with $-\E(A)$. Thus it acts cyclically on $\pm \E(A)$.

Identified the vertex-set of $K_{2nk+1}$ with $\Z_{2nk+1}$, we define the map $\rho$ on the set of its oriented edges so that:
$$\rho((x,x+a))= (x,x+\bar{\rho}_0(a)).$$
Since $\bar{\rho}_0$ acts cyclically on $\pm \E(A)=\mathbb{Z}_{2nk+1}\setminus\{0\}$, the map $\rho$ is a rotation of $K_{2nk+1}$.
Hence, by Theorem \ref{thm:embedding}, there exists a cellular embedding $\sigma$ of $K_{2nk+1}$ in an
orientable surface (i.e. the Archdeacon embedding) so that the face boundaries correspond to the orbits of $\rho\circ \tau$ where $\tau((x,x+a))=(x+a,x)$.

Let us consider the oriented edge $(x,x+a)$ with $a \in \E(A)$, and let ${C}$ be the column containing $a$.
Moreover, let us denote by $\lambda_c$ the minimum positive integer such that $\sum_{i=0}^{\lambda_c|\E({C})|-1} \omega_c^i(a)=0$ in $\mathbb{Z}_{2nk+1}$.
Since $a\in \E(A)$, $-a\in -\E(A)$ and we have that:
$$\rho\circ\tau((x,x+a))=\rho((x+a,(x+a)-a))=(x+a,x+a+\omega_c(a)).$$
Thus $(x,x+a)$ belongs to the boundary of the face $F_1$ delimited by the oriented edges:
$$(x,x+a),(x+a,x+a+\omega_c(a)),(x+a+\omega_c(a),x+a+\omega_c(a)+\omega_c^2(a)),\dots$$
$$\dots ,\left(x+\sum_{i=0}^{\lambda_c|\E({C})|-2} \omega_c^i(a),x\right).$$
We note that the circuit associated to the face $F_1$ is:
$$\left(x,x+a,x+a+\omega_c(a),\ldots,x+\sum_{i=0}^{\lambda_c|\E({C})|-2} \omega_c^i(a)\right).$$
Let us now consider the oriented edge $(x,x+a)$ with $a\not \in \E(A)$.
Hence $-a\in \E(A)$, and we name by ${R}$ the row containing the element $-a$.
Moreover, let us denote by $\lambda_r$ the minimum positive integer such that $\sum_{i=0}^{\lambda_r|\E({R})|-1} \omega_c^i(-a)=0$ in $\mathbb{Z}_{2nk+1}$.
Since $-a\in \E(A)$ we have
that:
$$\rho\circ\tau((x,x+a))=\rho((x+a,(x+a)-a))=(x+a,x+a-\omega_r(-a)).$$
Thus $(x,x+a)$ belongs to the boundary of the face $F_2$ delimited by the oriented edges:
$$(x,x+a),(x-(-a),x-(-a)-\omega_r(-a)),$$
$$(x-(-a)-\omega_r(-a),x-(-a)-\omega_r(-a)-\omega_r^2(-a)),\dots, \left(x-\sum_{i=0}^{\lambda_r|\E({R})|-2}
\omega_r^i(-a),x\right).$$
Since $A$ is a non-zero sum Heffter array and $\omega_r$ acts cyclically on $\E({R})$, for any $j\in [1,
\lambda_r|\E({R})|]$ we have that:
$$-\sum_{i=0}^{j-1}\omega_{r}^{i}(-a)=
\sum_{i=j}^{\lambda_r|\E({R})|-1}\omega_{r}^{i}(-a)=\sum_{i=1}^{\lambda_r|\E({R})|-j}\omega_{r}^{\lambda_r|\E({R})|-i}
(-a)=\sum_{i=1}^{\lambda_r|\E({R})|-j}\omega_{r}^{-i}(-a).$$
It follows that the circuit associated to the face $F_2$ can be written also as:
$$\left(x,x+\sum_{i=1}^{\lambda_r|\E({R})|-1}\omega_{r}^{-i}(-a),x+\sum_{i=1}^{\lambda_r|\E({R})|-2}\omega_{r}^{-i}(-a),
\dots,x+\omega_{r}^{-1}(-a)\right).$$
Therefore any nonoriented edge $\{x,x+a\}$ belongs to the boundaries of exactly two faces: one of type $F_1$ and one of
type $F_2$. Hence the embedding is 2-colorable.

Moreover, it is easy to see that those face boundaries are the circuits of the decompositions $\mathcal{C}(\mathcal{D}_{\omega_r^{-1}})$ and
$\mathcal{C}(\mathcal{D}_{\omega_c})$.
\end{proof}
\begin{rem}
Note that the proof of Theorem \ref{thm:biembedding} can be adapted to the case of a $^\lambda\N\H_t(m,n;h,k)$ that is simple with respect to the compatible orderings $\omega_r$ and $\omega_c$. Here we obtain a cellular biembedding of the circuit decompositions $\mathcal{C}(\mathcal{D}_{\omega_r^{-1}})$ and
$\mathcal{C}(\mathcal{D}_{\omega_c})$ of $^\lambda K_{\frac{2nk}{\lambda}+t}$ into an orientable surface. However, since the notations would get much more complicated and since our existence results of Theorems \ref{thm:gs} and \ref{thm:gsrectangular} cover the case $t=\lambda=1$ we have written the proof only in that case.
\end{rem}

We note that the existence results on compatible orderings of \cite{CDP} do not depend on the sum of the elements in each row and each column of the array. Hence, from Proposition $4.8$, Theorem $4.12$ and Theorem $4.15$ of \cite{CDP}, it follows that, given a cyclically $k$-diagonal non-zero sum $\N\H(n;k)$, there exists a pair of compatible orderings in the following cases:
\begin{itemize}
\item[(1)] if $nk$ is odd and $gcd(n,k-1)=1$;
\item[(2)] if $nk$ is odd and $3<k<200$;
\item[(3)] if $nk$ is odd and $n\geq (k-2)(k-1)$.
\end{itemize}
Furthermore, due to Theorem \ref{thm:gs}, for all such values of $n$ and $k$, there exists a biembedding of $K_{2nk+1}$ into an orientable surface whose face lengths are multiples of $k$. It follows that:
\begin{thm}\label{existenceBiemb1}
Let $n\geq k$ be odd integers such that one of the following holds:
\begin{itemize}
\item[(1)] $gcd(n,k-1)=1$;
\item[(2)] $3<k<200$;
\item[(3)] $n\geq (k-2)(k-1)$.
\end{itemize}
Then there exists a biembedding of $K_{2nk+1}$ into an orientable surface $\Sigma$ whose face lengths are multiples of $k$.
\end{thm}

Moreover, from Theorem $3.3$ of \cite{CDP}, it follows that given a $\N\H(m,n;n,m)$, there exists a pair of compatible orderings whenever $m$ and $n$ are not both odd.
As a consequence of Theorem \ref{thm:gsrectangular}, we have that:
\begin{thm}\label{existenceBiemb2}
Let $m$ and $n$ be integers and assume that $m$ and $n$ are not both odd.
Then there exists a biembedding of $K_{2mn+1}$ into an orientable surface $\Sigma$ whose face lengths are multiples of $m$ or multiples of $n$.
\end{thm}
\begin{rem}\label{genus}
Let us assume that we are under the hypothesis of Theorem \ref{existenceBiemb1} and that $2nk+1$ is a prime. Then each face is a circuit of length $k(2nk+1)$ and hence we have $2n$ faces. It follows that the genus of the surface $\Sigma$ is:
$$g=\frac{2-(2nk+1)+(2nk+1)nk-2n}{2}.$$
Similarly, if we are under the hypothesis of Theorem \ref{existenceBiemb2} and $2mn+1$ is a prime, then we have $m+n$ faces. It follows that the genus of the surface $\Sigma$ is:
$$g=\frac{2-(2mn+1)+(2mn+1)mn-(m+n)}{2}.$$
\end{rem}
\section*{Acknowledgements}
The first and the third authors were partially supported by INdAM--GNSAGA.


\begin{thebibliography}{50}
\bibitem{AB} R.J.R. Abel \and M. Buratti,
\textit{Difference families},
in: \textit{Handbook of Combinatorial Designs}. Edited by C. J. Colbourn and J. H. Dinitz. Second edition. Discrete
Mathematics and its Applications. Chapman \& Hall/CRC, Boca Raton, 2007.


\bibitem{AL} B. Alspach \and G. Liversidge,
\textit{On strongly sequenceable abelian groups},
Art Discrete Appl. Math. \textbf{3} (2020),  \#P1.02.

\bibitem{A} D.S. Archdeacon,
\textit{Heffter arrays and biembedding graphs on surfaces},
Electron. J. Combin. \textbf{22} (2015) \#P1.74.

\bibitem{ABD} D.S. Archdeacon, T. Boothby \and J.H. Dinitz,
\textit{Tight Heffter arrays exist for all possible values},
J. Combin. Des. \textbf{25} (2017), 5--35.

\bibitem{ADDY} D.S. Archdeacon, J.H. Dinitz, D.M. Donovan \and E.S. Yaz\i c\i,
\textit{Square integer Heffter arrays with empty cells},
Des. Codes Cryptogr. \textbf{77} (2015), 409--426.

\bibitem{ADMS} D.S. Archdeacon, J.H. Dinitz, A. Mattern \and D.R. Stinson,
\textit{On partial sums in cyclic groups},
J. Combin. Math. Combin. Comput. \textbf{98} (2016), 327--342.

\bibitem{BH} J.P. Bode \and H. Harborth, \textit{Directed paths of diagonals within polytopes}, Discrete Math. \textbf{299} (2005), 3--10.

\bibitem{BEZ} D. Bryant \and  S. El-Zanati,
\textit{Graph decompositions},
in: \textit{Handbook of Combinatorial Designs}. Edited by C. J. Colbourn and J. H. Dinitz. Second edition. Discrete
Mathematics and its Applications. Chapman \& Hall/CRC, Boca Raton, 2007.

\bibitem{B98} M. Buratti,
\textit{Recursive constructions for difference matrices and relative difference families},
J. Combin. Des. \textbf{6} (1998), 165--182.

\bibitem{BP} M. Buratti \and A. Pasotti,
\textit{Graph decompositions with the use of difference matrices},
Bull. Inst. Combin. Appl. \textbf{47} (2006), 23--32.


\bibitem{BCP} A.C. Burgess, N.J. Cavenagh, D.A. Pike,
\emph{Mutually orthogonal cycle systems}, preprint.


\bibitem{BCDY} K. Burrage, N.J. Cavenagh, D. Donovan \and  E.\c{S}. Yaz\i c\i,
\textit{Globally simple Heffter arrays $H(n;k)$ when $k\equiv 0,3 \pmod{4}$},
Discrete Math. \textbf{343} (2020), 111787.

\bibitem{CY1} Y. Caro \and R. Yuster,
\textit{Orthogonal decomposition and packing of complete graphs},
J. Combin. Theory Ser. A \textbf{88} (1999), 93--111.

\bibitem{CDDY} N.J. Cavenagh, J. Dinitz, D. Donovan \and  E.S. Yaz\i c\i,
\textit{The existence of square non-integer Heffter arrays},
Ars Math. Contemp. \textbf{17} (2019), 369--395.

\bibitem{CDY} N.J. Cavenagh, D. Donovan \and  E.\c{S}. Yaz\i c\i,
\textit{Biembeddings of cycle systems using integer Heffter arrays},
J. Combin. Des. \textbf{28} (2020), 900--922.


\bibitem{CDP} S. Costa, M. Dalai \and A. Pasotti,
\textit{A tour problem on a toroidal board},
Austral. J. Combin. \textbf{76} (2020), 183--207.

\bibitem{RelH} S. Costa, F. Morini, A. Pasotti \and M.A. Pellegrini,
\textit{A generalization of Heffter arrays},
J. Combin. Des. \textbf{28} (2020), 171--206.

\bibitem{CMPPSums} S. Costa, F. Morini, A. Pasotti \and  M.A. Pellegrini,
\textit{A problem on partial sums in abelian groups},
Discrete Math. \textbf{341} (2018), 705--712.

\bibitem{CMPPHeffter} S. Costa, F. Morini, A. Pasotti \and M.A. Pellegrini,
\textit{Globally simple Heffter arrays and orthogonal cyclic cycle decompositions},
Austral. J. Combin. \textbf{72} (2018), 549--593.

\bibitem{CPPBiembeddings} S. Costa, A. Pasotti \and M.A. Pellegrini,
\textit{Relative Heffter arrays and biembeddings},
 Ars Math. Contemp. \textbf{18} (2020), 241--271.

 \bibitem{CPEJC} S. Costa \and A. Pasotti, \textit{On $\lambda$-fold relative Heffter arrays and biembedding multigraphs on surfaces},
Europ. J. Combin. \textbf{97} (2021), 103370.


 \bibitem{CP} S. Costa \and M.A. Pellegrini, \textit{Some new results about a conjecture
by Brian Alspach}, Archiv der Mathematik \textbf{115} (2020), 479--488.


\bibitem{DM} J.H. Dinitz \and A.R.W. Mattern,
\textit{Biembedding Steiner triple systems and $n$-cycle systems on orientable surfaces},
Austral. J. Combin. \textbf{67} (2017), 327--344.

\bibitem{DW} J.H. Dinitz \and I.M. Wanless,
\textit{The existence of square integer Heffter arrays},
Ars Math. Contemp. \textbf{13} (2017), 81--93.



\bibitem{GT} J.L. Gross \and T.W. Tucker,
\textit{Topological Graph Theory},
John Wiley, New York, 1987.


\bibitem{HOS} J. Hicks,  M.A. Ollis \and J.R. Schmitt,
\textit{Distinct partial sums in cyclic groups: polynomial method and constructive approaches},
J. Combin. Des. \textbf{27} (2019), 369--385.


\bibitem{Moh} B. Mohar,
\textit{Combinatorial local planarity and the width of graph embeddings},
Canad. J. Math. \textbf{44} (1992), 1272--1288.

\bibitem{MT} B. Mohar \and C. Thomassen,
\textit{Graphs on surfaces},
Johns Hopkins University Press, Baltimore, 2001.


\bibitem{MP} F. Morini \and M.A. Pellegrini,
\textit{On the existence of integer relative Heffter arrays},
 Discrete Math. \textbf{343} (2020), 112088.

 \bibitem{MP2} F. Morini \and M.A. Pellegrini,
 \emph{Magic rectangles, signed magic arrays and integer $\lambda$-fold relative Heffter
arrays}, Austral. J. Combin. \textbf{80} (2021), 249--280.

 \bibitem{MP3} F. Morini \and M.A. Pellegrini,
 \emph{Rectangular Heffter arrays: a reduction theorem}, preprint available at https://arxiv.org/abs/2107.08857v2.

\bibitem{O} M.A. Ollis,
\textit{Sequences in dihedral groups with distinct partial products},
Austral. J. Combin. \textbf{78} (2020), 35--60.


\end{thebibliography}
\end{document}